\documentclass{article}
\usepackage[top=1in,bottom=1.1in,left=1in,right=1in]{geometry}
\usepackage{amssymb}
\usepackage{amsmath,color}
\usepackage{amsthm}
\usepackage{verbatim}
\usepackage{hyperref}
\usepackage[T1]{fontenc}
\usepackage{graphicx}

\newtheorem{theorem}{Theorem}[section]
\newtheorem{proposition}[theorem]{Proposition}
\newtheorem{corollary}[theorem]{Corollary}
\newtheorem{example}[theorem]{Example}

\newtheorem{lemma}[theorem]{Lemma}
\theoremstyle{definition}
\newtheorem{definition}[theorem]{Definition}

\newcommand{\abs}[1]{\left\vert#1\right\vert}     
\newcommand{\CAT}{\operatorname{CAT}}

\newcommand{\UC}{\operatorname{UC}}
\newcommand{\ULC}{\operatorname{ULC}}
\newcommand{\N}{\mathbb{N}}                                            
\newcommand{\R}{\mathbb{R}}                                         
\newcommand{\eps}{\varepsilon} 

\begin{document}

\title{``Lion-Man'' and the  Fixed Point Property}

\author{Genaro L\'{o}pez-Acedo$^{a}$, Adriana Nicolae$^{b,a}$, Bo\.{z}ena Pi\k{a}tek$^{c}$}
\date{}
\maketitle

\begin{center}
{\footnotesize
$^{a}$Department of Mathematical Analysis - IMUS, University of Seville, Sevilla, Spain\\
$^{b}$Department of Mathematics, Babe\c s-Bolyai University, Kog\u alniceanu 1, 400084 Cluj-Napoca, Romania\\
$^{c}$Institute of Mathematics, Silesian University of Technology, 44-100 Gliwice, Poland\\
\ \\
E-mail addresses:  glopez@us.es (G. L\'{o}pez-Acedo), anicolae@math.ubbcluj.ro (A. Nicolae), Bozena.Piatek@polsl.pl (B. Pi\k{a}tek).
}
\end{center}

\begin{abstract}
This paper focuses on the relation between the fixed point property for continuous mappings and a discrete lion and man game played in a strongly convex domain. Our main result states that in locally compact geodesic spaces, the compactness of the domain is equivalent to its fixed point property, as well as to the success of the lion. The common link among these properties involves the existence of different types of rays, which we also discuss.
\end{abstract}

{\small {\sl Keywords}: Locally compact geodesic space, lion and man game, fixed point property, compactness.} 

\section{Introduction}
Pursuit-evasion games go back a long way. Their origin could be placed in the fifth century BC when Zeno of Elea proposed his eternal paradoxes whose analysis led to frutiful theories in diverse mathematical areas (see \cite{Na07}). Among these games, one of the most challenging is Rado's famous lion and man problem (see \cite{Li86,Cr64}) which can be described as follows: a lion and a man move in a circular arena with equal maximum speeds. If the arena is viewed as a closed disc, the positions of the lion and of the man are regarded as two points and the lion must move so that the center of the disc, its position and the one of the man are collinear in this order, can the lion catch the man? The solution to this apparently easy problem was provided many years later by Besicovitch who showed that the man can escape if he follows a polygonal spiraling path. A detailed discussion of this solution can be found in \cite{Na07}. Related problems are obtained if for instance the man is confined to move on the boundary of the disc or one changes other conditions in the original game. Such a variation is a discrete game where the movement of both the lion and the man is limited to constant or bounded jumps. More recently, games of this type have also been approached in convex subsets of the sphere (see, e.g., \cite{KR05,AleBisGhr09}). Here we focus on a discrete-time equal-speed pursuit game considered in \cite{AleBisGhr06,AleBisGhr10}.
  
The domain $A$ of our game is a convex subset of a uniquely geodesic space $(X,d)$. Initially, the lion and the man are located at two points in $A$, $L_0$ and $M_0$, respectively. One fixes $D > 0$ an upper bound on the distance the lion and the man may jump. After $i$ steps, the man  moves from the point $M_{i}$ to any point $M_{i+1} \in A$ which is within distance $D$. The lion moves from the point $L_i$ to the point $L_{i+1}$ along the geodesic from $L_{i}$ to $M_{i}$ such that its distance to $L_{i}$ equals $\min\{D,d(L_i,M_i)\}$. We say that the lion wins if $\lim_{i\to\infty} d(L_{i+1},M_i) = 0$. When we refer in the sequel to the Lion-Man game, we will always mean the game we have just described. In \cite{AleBisGhr10}, it was stated that in CAT$(0)$ spaces the lion always wins if and only if the domain is compact. Nonetheless, this characterization of compactness proved to be false as \cite{Bac12} contains an example of an unbounded CAT$(0)$ space where the lion always wins. 

It is clear from the proof given in \cite{AleBisGhr10} that the solution of this game is deeply connected with the construction of geodesic rays in the domain $A$ and so are its compactness and the fixed point property for continuous mappings (see \cite{LopPia15,LopPia16}). Starting from this observation, we prove the equivalence between these seemingly unrelated properties.

\begin{theorem}\label{mainthm}
Let $A$ be a nonempty, closed and strongly convex subset of a complete, locally compact, uniquely geodesic space. Then the following are equivalent:
\begin{itemize}
\item[(i)] $A$ is compact;
\item[(ii)] $A$ has the fixed point property for continuous mappings;
\item[(iii)] the lion always wins the Lion-Man game played in $A$. 
\end{itemize}
\end{theorem}

In Section \ref{section-rays} we give basic definitions and properties of geodesic spaces which constitute a natural generalization of Riemannian manifolds. In particular, we fix in this setting the meaning of some convexity concepts such as the betweenness property or the strong convexity of a set in the form it was done in \cite{Cha93} for Riemannian manifolds. We show that for convex sets, strong convexity is equivalent to the betweenness property, a fact that will be essential in the study of the Lion-Man game. It is clear that the existence of a geodesic ray in the domain implies that the lion cannot always win the Lion-Man game because he and the man could move along this ray maintaining their distance constant. At the same time, the existence of a closed topological ray in the domain implies, as a consequence of the Tietze extension theorem, the failure of the fixed point property for continuous mappings. The main contribution of this section is the analysis of the existence and the relation among different types of rays: topological, polygonal and geodesic.

Section \ref{fpp} deals with the fixed point property for continuous mappings in geodesic spaces. We say that a subset $A$ of a topological space has the fixed point property if any continuous self-mapping defined on $A$ has at least one fixed point. Klee gave in \cite{klee} a characterization of compactness for convex subsets of a locally convex linear space by means of the fixed point property for continuous mappings. We prove a counterpart of this result in the setting of geodesic spaces, Theorem \ref{characfpp-thm}, which states that a closed convex subset of a complete, locally compact, uniquely geodesic space is compact if and only if it has the fixed point property for continuous mappings.

In Section \ref{lion-game} we analyze the Lion-Man game. After carefully fixing the rules of the game and explaining via Example \ref{ex-solution} the definition of its solution, we prove the main result of this section, Theorem \ref{thm-compact-lion}, which characterizes compactness of closed and strongly convex sets in a complete, locally compact, uniquely geodesic space in terms of the success of the lion. Theorem \ref{mainthm} is solely a synthesis of Theorems  \ref{characfpp-thm} and \ref{thm-compact-lion}. Finally, we would like to mention another recent work \cite{Yuf18} that studies the solution of this game in compact geodesic spaces.

\section{Convexity and rays in geodesic spaces}\label{section-rays} 
We start by fixing notation and recalling some basic facts about geodesic spaces. A detailed discussion on geodesic spaces can be found in \cite{Bri99,Bur01,Pap05}. Let $(X,d)$ be a metric space. For $x \in X$ and $r > 0$, we denote the {\it closed ball} centered at $x$ with radius $r$ by $\overline{B}(x,r)$. If $A$ is a nonempty subset of $X$, the {\it diameter} of $A$ is $\text{diam}\; A = \sup\{d(a,a'): a,a' \in A\}$ and the {\it distance} of a point $x \in X$ to $A$ is $\text{dist}(x,A) = \inf\{d(x,a) : a \in A\}$. The {\it distance} between two nonempty subsets $A$ and $B$ of $X$ is given by $d(A,B) = \inf\{d(a,b):a \in A, b \in B\}$. 

Let $x,y \in X$. A {\it geodesic} joining $x$ and $y$ is a mapping $\gamma:[0,l]\subseteq {\mathbb R}\to X$ such that $\gamma(0)=x$, $\gamma(l)=y$ and $d(\gamma(s),\gamma(t))=\abs{s-t}$ for all $s,t\in[0,l]$. This immediately yields $l = d(x,y)$. The geodesic $\gamma$ can be linearly reparametrized by the interval $[0,1]$ to obtain a mapping $\gamma' : [0,1] \to X$, $\gamma'(t) = \gamma(tl)$ and in this case $\gamma'$ is called a {\it linearly reparametrized geodesic}. The image $\gamma([0,l])$ of a geodesic $\gamma$ is called a {\it geodesic segment} joining $x$ and $y$. If instead of the interval $[0,l]$ one considers $[0,\infty)$, then the image of $\gamma$ is called a {\it geodesic ray} (sometimes we also refer to the mapping $\gamma$ as a geodesic ray). A point $z\in X$ belongs to a geodesic segment joining $x$ and $y$ if and only if there exists $t\in [0,1]$ such that $d(z,x)= td(x,y)$ and $d(z,y)=(1-t)d(x,y)$. In general, geodesic segments between two fixed points may not be unique. Whenever there is a unique geodesic between $x$ and $y$, we denote the unique geodesic segment joining them by $[x,y]$. We consider in the sequel the framework of metric spaces where every two points are joined by a (unique) geodesic. Such spaces are called {\it (uniquely) geodesic spaces}. We state next a property which we will use in the subsequent section and can be found in \cite[Chapter I, Lemma 3.12]{Bri99}.

\begin{lemma} \label{lemma-geod-unif-conv}
Let $X$ be a complete, locally compact, uniquely geodesic space. Suppose that $\gamma, \gamma_n : [0,1] \to X$ are linearly reparametrized geodesics such that the sequences $(\gamma_n(0))$ and $(\gamma_n(1))$ converge to $\gamma(0)$ and $\gamma(1)$, respectively. Then $(\gamma_n)$ converges uniformly to $\gamma$. 
\end{lemma}

Although in our main results there is no need to impose any curvature bounds on the space, nonpositively curved spaces in the sense of Busemann and Alexandrov spaces of curvature bounded above or below constitute relevant settings where rather general properties that we discuss below hold. Note that in what follows the curvature bounds are considered globally.

Let $(X,d)$ be a geodesic metric space. We say that $X$ is {\it nonpositively curved in the sense of Busemann} if given any two geodesics $\gamma:[0,l] \to X$ and $\gamma':[0,l'] \to X$,
\[d(\gamma(tl),\gamma'(tl')) \le (1-t)d(\gamma(0),\gamma'(0)) + td(\gamma(l),\gamma'(l')) \quad \text{for any } t \in [0,1].\]
Every nonpositively curved space in the sense of Busemann is uniquely geodesic and satisfies, in particular, the following property: for every $x \in X$ and every geodesic $\gamma:[0,l] \to X$ we have
\begin{equation}\label{eq-metric-conv}
d(x, \gamma(tl)) \le (1-t)d(x,\gamma(0))+td(x,\gamma(l)) \quad\text{for any } t \in [0,1].
\end{equation}

In the following we consider $(X,d)$ a geodesic space and let, for $\kappa \in \mathbb{R}$, $M^2_{\kappa}$ be the complete, simply connected $2$-dimensional Riemannian manifold of constant sectional curvature $\kappa$. Moreover, let $D_\kappa$ denote the diameter of $M_\kappa^2$, that is, $D_\kappa=\infty$ for $\kappa\le 0$ and $D_\kappa=\pi/\sqrt{\kappa}$ for $\kappa>0$. A {\it geodesic triangle} $\Delta = \Delta(x_1,x_2,x_3)$ in $X$ consists of three points $x_1, x_2,x_3 \in X$ (its {\it vertices}) and three geodesic segments (its {\it sides}) joining each pair of points. A {\it comparison triangle} for $\Delta$ is a triangle $\overline{\Delta} = \Delta(\overline{x}_1, \overline{x}_2, \overline{x}_3)$ in $M^2_{\kappa}$ with $d(x_i,x_j) = d_{M^2_{\kappa}}(\overline{x}_i,\overline{x}_j)$ for $i,j \in \{1,2,3\}$. For $\kappa$ fixed, comparison triangles of geodesic triangles always exist and are unique up to isometry. 

A geodesic triangle $\Delta$ satisfies the {\it $\CAT(\kappa)$ inequality} if for every comparison triangle $\overline{\Delta}$ in $M^2_{\kappa}$ of $\Delta$ and for every $x,y \in \Delta$ we have
\[d(x,y) \le d_{M^2_{\kappa}}(\overline{x},\overline{y}),\]
where  $\overline{x},\overline{y} \in \overline{\Delta}$ are the comparison points of $x$ and $y$, i.e., if $x$ belongs to the side joining $x_i$ and $x_j$, then $\overline{x}$ belongs to the side joining  $\overline{x}_i$ and $\overline{x}_j$ such that $d(x_i,x) = d_{M^2_{\kappa}}(\overline{x}_i,\overline{x})$.

A {\it $\CAT(\kappa)$ space} is a metric space where each two points at distance less than $D_\kappa$ are joined by a geodesic and where every geodesic triangle having perimeter less than $2D_\kappa$ satisfies the CAT$(\kappa)$ inequality. $\CAT(\kappa)$ spaces are also known as spaces of {\it curvature bounded above} by $\kappa$ (in the sense of Alexandrov). In any $\CAT(\kappa)$ space there exists a unique geodesic joining each pair of points at distance less than $D_\kappa$. A geodesic space is said to have {\it curvature bounded below} by $\kappa$ (in the sense of Alexandrov) if every geodesic triangle in it having perimeter less than $2D_\kappa$ satisfies the reverse of the $\CAT(\kappa)$ inequality. When referring to these spaces, we sometimes omit the bound $\kappa$ if it can be chosen arbitrarily. 

Hilbert spaces constitute a prime example of a geodesic spaces that have curvature bounded both above and below by $0$. The complex Hilbert ball with the hyperbolic metric is another example of a CAT$(0)$ space which also has curvature bounded below. Other important examples of CAT$(0)$ spaces include, e.g., Hadamard manifolds or Euclidean buildings. Among spaces of curvature bounded below by $0$ one finds, e.g., complete Riemannian manifolds of nonnegative sectional curvature or convex surfaces in $\R^3$ with the induced metric. We refer to \cite{Bri99,Bur01} for more details.

An {\it ${\R}$-tree} is a uniquely geodesic space $X$ such that if $x,y,z \in X$ with $[y,x]\cap [x,z]=\{x\}$, then $[y,x]\cup [x,z]=[y,z]$. It is easily seen that a metric space is an ${\R}$-tree if and only if it is a $\CAT(\kappa)$ space for any real $\kappa$. 

Although some of the concepts and properties below can also be given without assuming uniqueness of geodesics, because in our main results we use this condition, we assume for simplicity in the rest of this section that $(X,d)$ is  a uniquely geodesic space. In order to introduce and relate several convexity notions, we first recall that a geodesic $\gamma : [0, l] \to X$ is said to be {\it extendable} beyond the point $\gamma(l)$ if $\gamma$ is the restriction of a geodesic $\gamma' : [0, l'] \to X$ with $l' > l$. 

We say that two geodesics {\it bifurcate} if they have a common endpoint and coincide on an interval, but one is not an extension of the other. Geodesic spaces of curvature bounded below cannot have bifurcating geodesics. Since geodesics in $\R$-trees can bifurcate, they are not spaces of curvature bounded below.

A {\it local geodesic} is a mapping $\gamma:[0,l]\subseteq {\mathbb R}\to X$ with the property that for every $t \in [0,l]$ there exists a nontrivial closed interval $I$ containing $t$ in its interior such that $\left.\gamma\right|_{I \cap [0,l]}$ is a geodesic. Note that in a nonpositively curved space in the sense of Busemann, any local geodesic is in fact a geodesic (see, e.g., \cite[Corollary 8.2.3]{Pap05}). Likewise, in any CAT$(\kappa)$ space, every local geodesic of length at most $D_\kappa$ is a geodesic (see, e.g., \cite[Chapter II, Proposition 1.4]{Bri99}). 

We will use the following two convexity notions. For corresponding definitions in the Riemannian setting, we refer, e.g., to \cite[Chapter IX.6]{Cha93}.  
\begin{definition}
Let $A$ be a set in $X$. We say that $A$ is {\it convex} if given any two points $x, y \in A$, the geodesic segment $[x,y]$ is contained in $A$. The set $A$ is said to be {\it strongly convex} if for every $x, y \in A$, the geodesic segment $[x,y]$ is contained in $A$ and there is no other local geodesic in $A$ joining $x$ and $y$. 
\end{definition}
Note that there exist convex sets that are not strongly convex (see, e.g., \cite{Ka}).

We recall next a betweenness property which also appears in \cite{Pap05} and plays an essential role in the study of the considered pursuit-evasion problem. 
\begin{definition}
Let $A$ be a set in $X$. We say that $A$ satisfies the {\it betweenness property} if for every four pairwise distinct points $x, y, z, w \in A$, if $y \in [x,z]$ and $z \in [y,w]$, then $y,z \in [x,w]$.
\end{definition}
The betweenness property is satisfied by all sets in nonpositively curved spaces in the sense of Busemann (see \cite[Proposition 8.2.4]{Pap05}). Additional examples of geodesic spaces where this holds have been given in \cite{Nic13}: one can assume that $X$ satisfies \eqref{eq-metric-conv} (see \cite[Proposition 3.4]{Nic13}) or that geodesics in $X$ are extendable and do not bifurcate (which, as mentioned before, happens in spaces of curvature bounded below) (see \cite[Proposition 3.5, Corollary 3.6]{Nic13}). 

\begin{proposition} \label{prop-btw-local-geod}
Let $A$ be a convex set in $X$. Then the following statements are equivalent:
\begin{itemize}
\item [(i)] $A$ is strongly convex;
\item[(ii)] every local geodesic in $A$ is a geodesic;
\item[(iii)] $A$ has the betweenness property.
\end{itemize}
\end{proposition}
\begin{proof}
The equivalence between (i) and (ii) is obvious. 

We prove next that (ii) implies (iii). Let $x,y,z,w$ be pairwise distinct points in $A$ with $y \in [x,z]$ and $z \in [y,w]$. The union of $[x,z]$ and $[y,w]$ is the image of a local geodesic $\gamma$ with $\gamma(0) = x$ and $\gamma(l) = w$ that contains $y$ and $z$. By (ii), $\gamma$ is actually the unique geodesic joining $x$ and $w$, so $y,z \in [x,w]$.

For the converse implication, let $\gamma : [0,l] \to A$ be a local geodesic and $t_0 \in [0,l]$ be maximal such that $\left.\gamma\right|_{[0,t_0 ]} \text{ is a geodesic}$. Suppose that $t_0 < l$. Since $\gamma$ is a local geodesic, there exists $\eps \in (0,t_0)$ with $t_0 + \eps < l$ such that $\left.\gamma\right|_{[t_0-\eps,t_0+\eps]}$ is a geodesic. We can now apply the betweenness property to the respective points $\gamma(0)$, $\gamma(t_0 - \eps)$, $\gamma(t_0)$ and $\gamma(t_0 + \eps)$ to conclude that $\left.\gamma\right|_{[0,t_0+\eps]}$ is a geodesic, which contradicts the maximality of $t_0$. \qed
\end{proof}

Thus, if $X$ is a nonpositively curved space in the sense of Busemann or a $\CAT(\kappa)$ space of diameter less than $D_k$, then any convex subset of it is in fact strongly convex. Moreover, this is still true when $X$ is a uniquely geodesic space that either satisfies \eqref{eq-metric-conv} or has extendable geodesics and curvature bouneded below (or, more generally, nonbifurcating geodesics). 

As before, $(X,d)$ stands in the following for a uniquely geodesic space. A {\it topological ray} in $X$ is a homeomorphic image of the interval $[0,\infty)$. We say that a topological ray is {\it closed} if it is closed as a subset of $X$. Suppose $L$ is a closed topological ray and denote by $h$ such  a homeomorphism whose image is $L$. Then $L$ is called a {\it polygonal ray} if there exists an unbounded strictly increasing sequence of real numbers $0=v_0 < v_1 < \ldots < v_n < \ldots$ such that $h([v_n,v_{n+1}])$ is a geodesic segment for every $n \in \mathbb{N}$. Obviously, every geodesic ray is a polygonal ray which, in its turn, is a closed topological ray by definition. On the other hand, as pointed out below, one can use the fixed point property for continuous mappings to deduce the existence of geodesic rays from the existence of closed topological ones.

If $X$ contains a closed topological ray $\Gamma$, then it does not have the fixed point property for continuous mappings. To see this, let $h:[0,\infty) \to \Gamma$ be a homeomorphism. By the Tietze extension theorem, $h^{-1}$ can be extended to a continuous $H : X  \to [0,\infty)$. Clearly, there exists a fixed point free continuous mapping $f : \Gamma \to \Gamma$. Then $f \circ h \circ H : X \to \Gamma$ is continuous, has no fixed points and regarding it as taking values in $X$, this implies that $X$ does not have the fixed point property for continuous mappings. 

By \cite[Theorem 3.4]{Ki04}, a complete $\R$-tree has the fixed point property for continuous mappings if and only if it does not contain a geodesic ray. Using the above remark, this immediately yields that in complete $\R$-trees, the existence of geodesic rays is in fact equivalent to the existence of closed topological rays. Consider the $\mathbb{R}$-tree obtained by endowing $\R^2$ with the so-called river metric $d_r$ defined by
$$
d_r((x_1,x_2),(y_1,y_2))=\left\{
\begin{array}{ll}
|x_2 - y_2| & \mbox{if } x_1 = y_1 \\[1em]
|x_2|+|y_2|+|x_1-y_1| & \mbox{otherwise}.
\end{array}\right.
$$
This $\R$-tree is complete and nonlocally compact. Its closed unit ball is noncompact and contains no geodesic rays, hence no closed topological rays. In the same line, an example of a complete noncompact CAT$(0)$ space without polygonal rays is constructed in \cite{LopPia15}. Such examples cease to exist when assuming that the space has curvature bounded below or is locally compact. More precisely, by \cite[Theorem 8]{LopPia15}, any closed, convex and noncompact subset of a complete uniquely geodesic space of curvature bounded below contains a polygonal ray. Furthermore, the following result, which is Theorem 3.2 in \cite{LopPia16}, is a direct consequence of the Arzela-Ascoli and the Hopf-Rinow theorems. 

\begin{theorem} \label{thm-noncompact-ray}
A closed and convex subset of a complete, locally compact, uniquely geodesic space is noncompact if and only if it contains a geodesic ray.
\end{theorem}

We say that a polygonal ray $L$ is {\it inscribed} in a topological ray $\Gamma$ if there exists $(a_n) \subseteq \Gamma$ such that $L = \bigcup_{n \ge 0}[a_n,a_{n+1}]$. We end this section with the following observation concerning topological and polygonal rays, which, although not used in this work, has a value on its own. 

\begin{proposition}
Let $X$ be a complete uniquely geodesic space of curvature bounded below and let $\Gamma$ be a closed topological ray in $X$. Then there exists a polygonal ray inscribed in $\Gamma$.
\end{proposition}
\begin{proof}
Let $h:[0,\infty) \to \Gamma$ be a homeomorphism. Suppose first that $\Gamma$ is bounded. Take $a_n = h(n)$, for $n \in \mathbb{N}$. Assume that $(a_n)$ has a convergent subsequence $(a_{n_k})$ whose limit belongs to $\Gamma$. Then $h^{-1}(a_{n_k}) = n_k$, which, by passing to limit, yields a contradiction. Thus, $(a_n)$ has no Cauchy subsequence and therefore must contain a separated subsequence. Now we are in the same situation as in the proof of \cite[Theorem 8]{LopPia15} and so one can construct the desired polygonal ray.

Suppose next $\Gamma$ is unbounded and let $t_0 = 0$. Then, for every $n \in \mathbb{N}$, there exists $t_{n+1} \ge t_n$ with $d(h(0),h(t_{n+1})) \ge n+1$. Clearly, $\lim_{n \to \infty} t_n = \infty$. If for each $n \in \mathbb{N}$ one can find $s_n > t_n$ such that the sequence $(h(s_n))$ is bounded, then one can reason as above in order to obtain the polygonal ray. Otherwise, let $a_0=h(0)$. Then there exists $n_0 \in \mathbb{N}$ such that for every $t > t_{n_0}$, $h(t) \notin \overline{B}\left(a_0,1\right)$. Let $C$ be the compact set $h\left([0,t_{n_0}]\right) \cap \overline{B}\left(a_0,1\right)$ and take $t'_0 \in [0,t_{n_0}]$ such that 
\[C=h\left([0,t'_0]\right) \cap \overline{B}\left(a_0,1\right),\] 
with $h(t'_0) \in \overline{B}\left(a_0,1\right)$. Denoting $a_1 = h(t'_0)$, we have that $d(a_0,a_1) = 1$. As before, considering $h|_{[t'_0,\infty)}$, one can find a point $a_2 \in \overline{B}\left(a_1,1\right)$ with $d(a_0,a_2) > 1$ and $d(a_1,a_2) = 1$ such that for any $t > h^{-1}(a_2)$, $h(t) \notin \overline{B}\left(a_1,1\right)$.  In this way one builds a sequence $(a_n)$ such that for each $n \in \mathbb{N}$ and $m \ge n+2$, $d(a_n,a_{n+1}) = 1$ and $d(a_m,a_n) > 1$. Moreover, because geodesics cannot bifurcate, 
\[[a_{n+1},a_n] \cap [a_{n+1},a_{n+2}] = \{a_{n+1}\},\]
for each $n \in \mathbb{N}$. Let $n \in \mathbb{N}$, $m \ge n+2$ and $x \in [a_m,a_{m+1}]$. By the triangle inequality, $d(x,a_n) \ge 1/2$ and $d(x, a_{n+1}) \ge 1/2$. Applying now \cite[Lemma 6]{LopPia15} to the points $x, a_n, a_{n+1}$, there exists $M > 0$ (not depending on either $x$, $n$ or $m$) such that $\text{dist}\left(x,[a_n,a_{n+1}]\right) \ge M$. Thus,  $\text{dist}\left([a_m,a_{m+1}],[a_n,a_{n+1}]\right) \ge M$. If $(y_i) \subseteq \bigcup_{n \ge 0}[a_n,a_{n+1}]$ is a convergent sequence, then there exists $N \in \mathbb{N}$ such that $y_i \in [a_N,a_{N+1}]$ for $i$ sufficiently large. This shows that $\bigcup_{n \ge 0}[a_n,a_{n+1}]$ is closed, which ends the proof. \qed
\end{proof}

\section{Fixed point property and compact sets}\label{fpp}
The well-known Schauder-Tychonoff theorem  asserts that a compact convex subset of a locally convex linear topological space has the fixed point property for continuous mappings.  Klee  \cite{klee} proved a converse of this result by using the fact that any noncompact convex subset of a locally convex metrizable linear topological space contains a closed topological ray and so lacks the fixed point property (see also \cite{Db97}). Counterparts of Klee's result in the framework of geodesic spaces have been recently proved by assuming either lower curvature bounds \cite{LopPia15} or local compactness \cite{LopPia16}.

The Schauder conjecture was a long-standing open question asking whether the local convexity condition can be omitted in the Schauder-Tychonoff theorem. A positive answer was given by Cauty \cite{Ca07,Ca17}. To introduce this result, we need to define first a particular class of contractible spaces (see \cite[p.~187]{Du65}).
 
\begin{definition}
A metric space $X$ is called {\it uniformly contractible} ($\UC$ for short) if there exists a continuous mapping $\lambda \colon X\times X\times [0,1]\to X$ such that 
\begin{equation}\label{def-eq-map}
\lambda(x,y,0)=x, \quad \lambda(x,y,1)=y, \quad \text{and }\quad \lambda(x,x,t)=x,
\end{equation}
for all $x,y,\in X$ and $t\in[0,1]$. 
\end{definition}

A more general notion is the following one.
\begin{definition}
A metric space $X$ is called {\it uniformly locally contractible} ($\ULC$ for short) if there exists a neighborhood $U$ of the diagonal of $X\times X$ and a mapping $\lambda: U\times [0,1] \to X$ satisfying \eqref{def-eq-map} for all $(x,y) \in U$ and $t\in [0,1]$.
\end{definition}

Let $X$ be a $\ULC$ space and $f : X\to X$. We say that $f$ is {\it compact} if $f(X)$ is contained in a compact subset of $X$. The next theorem was proved in \cite{Ca07} and provides a sufficient condition for a continuous and compact mapping $f$ to possess fixed points. This condition is given in terms of the Lefschetz number of $f$, which is denoted by $\Lambda(f)$ (see \cite[Chapter V.$\S$15]{Gr03} for more details).

\begin{theorem}\label{thm-Cauty}
Let $X$ be a $\ULC$ space and $f\colon X\to X$ continuous and compact. If $\Lambda(f) \not= 0$, then $f$ has a fixed point.
\end{theorem}

As a consequence, we obtain the following corollary that is essential for the proof of the main theorem in this section. In fact, this result generalizes Corollary 2.10 from \cite{LopPia16} by assuming local compactness instead of compactness and dropping the local convexity condition. 
\begin{corollary}\label{GenST}
Let $X$ be a complete, locally compact, uniquely geodesic space. If $f : X\to X$ is continuous and bounded, then $f$ has a fixed point. 
\end{corollary}

\begin{proof}
We show first that $X$ is a $\UC$ space. To see this, we define the mapping $\lambda\colon X\times X\times [0,1]\to X$ by taking $\lambda(x,y,t)$ the unique point on the geodesic segment $[x,y]$ such that $d(x,\lambda(x,y,t))=td(x,y)$. Then it is enough to prove that any sequence $(\lambda(x_n,y_n,t_n))$ converges to $\lambda(x,y,t)$ if $x_n\to x$, $y_n\to y$ and $t_n\to t$, which immediately follows by Lemma \ref{lemma-geod-unif-conv}. 

Thus, $X$ is contractible, so we can apply \cite[Lemma 3.2, Chapter V.$\S$15]{Gr03} to conclude that $\Lambda(f) = 1$. Moreover, $f$ is compact because $f(X) \subseteq \overline{f(X)}$, which is bounded and closed, hence compact by the Hopf-Rinow theorem. As $X$ is a $\ULC$ space, the result is now obtained from Theorem \ref{thm-Cauty}. \qed
\end{proof}

Obviously, Corollary \ref{GenST} holds true for each self-mapping defined on closed and convex subsets of $X$, so we can now state the main result of this section.
\begin{theorem}\label{characfpp-thm}
Let $(X,d)$ be a complete, locally compact, uniquely geodesic space $X$ and $A \subseteq X$ nonempty, closed and convex. Then $A$ has the fixed point property for continuous mappings if and only if $A$ is compact. 
\end{theorem}
\begin{proof}
If $A$ is compact, then any continuous self-mapping on $A$ is bounded, so we can apply Corollary \ref{GenST} to obtain that it has fixed points. Assume next that $A$ has the fixed point property for continuous mappings and suppose by contradiction that $A$ is not compact. By Theorem \ref{thm-noncompact-ray}, there exists a geodesic ray $\gamma : [0,\infty) \to A$. Define the mapping $f : A \to A$ by $f(x)=\gamma(d(\gamma(0),x)+1)$, $x \in A$. One can easily see that $f$ is fixed point free. Since 
\begin{align*}
d(f(x),f(y)) & = d(\gamma(d(\gamma(0),x)+1),\gamma(d(\gamma(0),y)+1))\\
& = |d(\gamma(0),x) - d(\gamma(0),y)| \le d(x,y),
\end{align*}
for all $x,y \in A$, $f$ is nonexpansive, hence continuous. This contradicts the fact that $A$ has the fixed point property for continuous mappings. Therefore $A$ must be compact. \qed
\end{proof}

A counterpart of this result can also be proved in the absence of local compactness when the geodesic space has curvature bounded below (see \cite[Theorem 10]{LopPia15}). For nonlocally compact spaces, the lower curvature bound cannot be dropped as follows from \cite[Theorem 3.4]{Ki04} mentioned in the previous section. 

\section{The Lion-Man game and compact sets}\label{lion-game}

Let $(X,d)$ be a uniquely geodesic space and $A \subseteq X$ nonempty and convex. Take $D > 0$ and suppose that $L_0, M_0 \in A$ are the starting points of the lion and the man, respectively. At step $i+1$, $i \in \N$, the lion moves from the point $L_i$ to the point $L_{i+1} \in [L_i,M_i]$ such that $d(L_i,L_{i+1}) = \min\{D,d(L_i,M_i)\}$. The man moves from the point $M_i$ to the point $M_{i+1} \in A$ satisfying $d(M_i,M_{i+1}) \le D$. We say that lion wins if the sequence $(d(L_{i+1},M_i))$ converges to $0$. Otherwise the man wins. Denote in the sequel $D_i = d(L_i,M_i)$, $i \in \N$.

It is easy to see that the lion wins if and only if either of the following two mutually exclusive situations holds:
\begin{itemize}
\item[(1)] there exists $i_0 \in \N$ such that $D_{i_0} \le D$. In this case, $L_{i+1} = M_i$ for all $i \ge i_0$;
\item[(2)] $D_i > D$ for all $i \in \N$ and $\lim_{i \to \infty}D_i = D$. Note that the last limit exists because in this case the sequence $(D_i)$ is nonincreasing as
\[D_{i+1} \le d(L_{i+1},M_i) + d(M_i,M_{i+1}) = D_i - D + d(M_i,M_{i+1}) \le D_i,\]
for all $i \in \N$.
\end{itemize}
Consequently, the man wins if and only if $D_i > D$ for all $i \in \N$ and $\lim_{i \to \infty}D_i > D$.

We start our discussion with an example which shows that even in compact convex subsets of the Euclidean plane, it is possible to construct games of this type where the lion wins by fulfilling conditions (1) and (2), respectively. Thus, we cannot define the success of the lion assuming just one of the conditions (1) or (2) if we aim to obtain a characterization of compact sets via this property. In the next example, which is inspired by the one due to Besicovitch, we use the notation introduced above.

\begin{example}\label{ex-solution}
Let $D > 0$ and $A = \overline{B}(O,7D)$, where $O$ is the origin in $\R^2$. Suppose that $L_0 = O$ and $M_0 \in A$ with $D < D_0 < 2D$. Furthermore, assume that at each step $i+1$, $i \in \N$, $\|M_i - M_{i+1}\| = D$ and the clockwise angle from $M_iL_i$ to $M_iM_{i+1}$ $\angle_{M_i}(L_i,M_{i+1}) = \pi/2$. 

Denoting for $i \in \N$, $t_{i+1} = \|L_{i+1} - M_i\|$, we have $(D + t_{i+2})^2 = D^2 + t^2_{i+1}$. Since $t_1 \in (0,D)$, this immediately yields $t_i \in  (0,D)$ for all $i \in \N$, so $2Dt_{i+2} + t^2_{i+2} < Dt_{i+1}$, from where $t_{i+2} < t_{i+1}/2$. Hence,
\begin{equation}\label{example-eq1}
\sum_{i \ge 1} t_i \le 2D.
\end{equation}
Clearly, condition (1) fails since $D_i = D + t_{i+1} > D$ for each $i \in \N$, but condition (2) is satisfied as $\lim_{i \to \infty}D_i = D$.

It remains to show that the lion and the man move indeed within the set $A$. To this end, denote $\alpha_{i+1} = \angle_{M_{i+1}}(L_{i+1},M_i)$, $i \in \N$. Then $t_{i+1} = D \tan \alpha_{i+1}$ and
\[\alpha_{i+1} = \arctan \frac{t_{i+1}}{D} \le \frac{t_{i+1}}{D}.\] 
For every $i \in \N$, take $B_i \in \R^2$ such that $B_iL_{i+1}M_iM_{i+1}$ is a rectangle. Then $\|L_{i+1}-B_i\| = \|L_{i+1} - L_{i+2}\| = D$ and $\angle_{L_{i+1}}(B_i,L_{i+2}) = \alpha_{i+1}$. Now for all $i \in \N$, let $C_i \in \R^2$ such that $C_iL_iL_{i+1}B_i$ is a square. Because $\|L_{i+1}-L_i\| = \|L_{i+1} - C_{i+1}\| = D$ and $\angle_{L_{i+1}}(L_i,C_{i+1}) = \alpha_{i+1}$ we obtain
\begin{equation}\label{example-eq2}
\|L_i - C_{i+1}\| = \|B_i - L_{i+2}\|  = 2D\sin\frac{\alpha_{i+1}}{2} \le t_{i+1}.
\end{equation}
We also have $\|L_{i+1} - C_i\| = \|L_{i+1} - B_{i+1}\| = \sqrt{2}D$ and $\angle_{L_{i+1}}(C_i,B_{i+1}) = \alpha_{i+1}$. Thus, 
\begin{equation}\label{example-eq3}
\|C_i - B_{i+1}\| = 2\sqrt{2}D \sin\frac{\alpha_{i+1}}{2} \le 2t_{i+1}.
\end{equation}
By \eqref{example-eq2} and \eqref{example-eq3},
\begin{align*}
\|L_{4k}-L_{4k+4}\| & \le \|L_{4k}-C_{4k+1}\| + \|C_{4k+1}-B_{4k+2}\| + \|B_{4k+2}-L_{4k+4}\| \\
& \le t_{4k+1} + 2t_{4k+2} + t_{4k+3},
\end{align*}
for all $k \in \N$. Therefore, by \eqref{example-eq1}, $\|L_0 - L_{4k}\| \le 2\sum_{i \ge 1}t_i \le 4D$ for all $k \in \N$, from where $\|L_0 - L_i\| \le 6D$ for all $i \in \N$. Moreover, 
\[\|L_0 - M_i\| \le \|L_0 - L_{i+1}\| + \|L_{i+1} - M_i\| \le 7D.\]

Note also that one can easily modify this example so that condition (1) holds: just consider that at the second step, the man changes the direction and moves from $M_1$ to $M_2 \in [M_1,L_1]$. In this case, $D_2 < D$. 
\end{example}

The main result of this section gives a characterization of compactness of a set where the Lion-Man game is played in terms of the success of the lion. We use again the notation introduced at the beginning of this section.

\begin{theorem}\label{thm-compact-lion}
Let $(X,d)$ be a complete, locally compact, uniquely geodesic space. Suppose $A \subseteq X$ is a nonempty, closed and strongly convex set where the Lion-Man game is played following the rules described above. Then $A$ is compact if and only if the lion always wins.
\end{theorem}
\begin{proof}
Assume first that $A$ is not compact. By Theorem \ref{thm-noncompact-ray}, there exists a geodesic ray $\gamma : [0, \infty) \to A$. Take $L_0 = \gamma(0)$ and $M_0 = \gamma(D+1)$. Then the man and the lion can move along this ray maintaining their distance $D_i = d(L_i,M_i)$ constantly equal to $D+1$. This means that the man wins.

We suppose next that $A$ is compact, hence bounded. Then there exists $N \in \N$ such that 
\begin{equation}\label{thm-compact-lion-eq1}
(N-1)D \le \text{diam}\; A < ND.
\end{equation}
Note that $N \ge 2$ (otherwise the game cannot be played in $A$). Assume by contradiction that the man wins. Then $D_i > D$ for all $i \in \N$ and 
\begin{equation}\label{thm-compact-lion-eq2}
\lim_{i \to \infty} D_i = D^* > D. 
\end{equation}

We show by induction that for any fixed $n \in \N$, the sequences $(L_i)$ and $(M_i)$ contain respective subsequences $(L_{i^n_j})_j$ and $(M_{i^n_j})_j$ (thus having both the same sequence of indices $(i^n_j)_j$ in $(L_i)$ and $(M_i)$) such that their respective limit points in $A$, $l_n$ and $m_n$, satisfy the following property: $l_n \in [l_0,m_n]$ with $d(l_0,l_n) = nD$. Note that, by \eqref{thm-compact-lion-eq2}, $d(l_n,m_n) = D^*$.

For $n = 0$ this is clear as $A$ is compact. Suppose that the above statement holds for $n=k$. We prove that it also holds for $n=k+1$. To this end, choose a convergent subsequence of $(L_{i^k_j +1})$ and denote its limit by $l_{k+1} \in A$. Since for all $j \in \N$, $L_{i^k_j + 1} \in [L_{i^k_j},M_{i^k_j}]$ with $d(L_{i^k_j},L_{i^k_j +1}) = D$, we have $d(l_k,l_{k+1}) = D$ and $d(l_{k+1},m_k) = D^* - D$, from where $l_{k+1} \in [l_k,m_k]$ with $d(l_k,l_{k+1}) = D$. This implies, by compactness of $A$, that the whole sequence $(L_{i^k_j + 1})$ converges to $l_{k+1}$. Furthermore, since $l_k \in [l_0,m_k]$ such that $d(l_0,l_k) = kD$, it follows that $l_{k+1} \in [l_0,m_k]$ and $d(l_0,l_{k+1}) = (k+1)D$.

Take a subsequence $(i^{k+1}_j)_j$ of $(i^k_j  + 1)_j$ such that $(M_{i^{k+1}_j}) \subseteq (M_{i^{k}_j + 1})$ is convergent and denote its limit by $m_{k+1} \in A$. In particular, the sequence $(L_{i^{k+1}_j})$ converges to $l_{k+1}$. Because $d(M_{i^k_j},M_{i^k_j +1}) \le D$ for all $j \in \N$, it follows that $d(m_k,m_{k+1}) \le D$. Furthermore, $d(l_{k+1},m_k) = D^* - D$ and $d(l_{k+1},m_{k+1}) = D^*$, so $m_k \in [l_{k+1}, m_{k+1}]$. Recalling that $l_{k+1} \in [l_0,m_k]$, by strong convexity of $A$ and Proposition \ref{prop-btw-local-geod}, we obtain $l_{k+1} \in [l_0,m_{k+1}]$. This finishes the induction argument.

Thus, there exist $l_0, l_N \in A$ such that $d(l_0,l_N) = ND$, which contradicts \eqref{thm-compact-lion-eq1} and therefore shows that the lion must win. \qed
\end{proof}

Now it is clear that Theorem \ref{mainthm} is a direct consequence of Theorems \ref{characfpp-thm}, and \ref{thm-compact-lion}. Moreover, by Theorem \ref{thm-noncompact-ray}, each of the statements in Theorem \ref{mainthm} is equivalent to the absence of geodesic rays from $A$.

The continuous Lion-Man game where the lion knows the position of the man at every time and moves directly towards it was addressed in $\CAT(k)$ spaces by Jun \cite{Jun14} who showed the existence of continuous pursuit curves and studied their regularity, as well as their approximation by discrete ones.

It should be noted that while, as far as we know, Theorem \ref{mainthm} is the first result which relates the Lion-Man game to the fixed point property for continuous mappings, considering the lion and the man moving in the closed unit disc, Bollob\'{a}s et al. \cite[Theorem 7]{BLW12} used the Brouwer fixed point theorem to prove that the man does not have a so-called continuous winning strategy. Using Corollary \ref{GenST} and following the same argument from \cite{BLW12}, one can prove an analogous result in a compact uniquely geodesic space. We also remark that \cite{BLW12} contains an analysis of the relation between the solution of the continuous game and its approximation by a discrete version. 

\section{Acknowledgement}
The authors would like to thank the referee for the comments and suggestions that helped improve the manuscript.

\end{document}